\documentclass{amsart}
\usepackage{amsmath,amssymb,amsfonts,euscript,graphicx}

\newtheorem{thm}{Theorem}[section]
\newtheorem{cor}[thm]{Corollary}
\newtheorem{lem}[thm]{Lemma}

\theoremstyle{definition}

\theoremstyle{remark}
\newtheorem{rem}[thm]{Remark}
\numberwithin{equation}{section}

\begin{document}

\title[]{some other algebraic properties of folded hypercubes  }%
\author{\bf S. Morteza. MIRAFZAL }%
\address{Department of Mathematics, University of Isfahan, Isfahan 81746-73441, Iran}%
\email{smortezamirafzal@yahoo.com}%
\email{}%


\begin{abstract} We construct explicity  the automorphism group of
the folded hypercube $FQ_n$ of dimension $n>3$,
as a semidirect product of $N$ by $M$, where $N$ is isomorphic to
the Abelian group $Z_2^n$ , and $M$ is isomorphic to $Sym(n+1)$, the
symmetric group of degree $n+1$,  then we will show that the folded
hypercube $FQ_n$ is a symmetric graph. \

\


\end{abstract} \maketitle%

\section{\bf Introduction and Preliminaries} A folded hypercube is an
edge transitive graph, this fact is the main result that has been
shown in $[8]$. In this note, we  construct explicity the
automorphism group of a folded hypercube, then we will show that a
folded hypercube is not only an edge transitive graph, but also a
symmetric graph. In this paper, a graph $G=(V,E)$ is considered as
an undirected graph where $V=V(G)$ is the vertex-set and $E=E(G)$ is
the edge-set. For all the terminology and notation not defined here,
we follow $[2,3,5]$. The hypercube $Q_n$ of dimension $n$ is the
graph with vertex-set $\{(x_1,x_2,...,x_n)| x_i \in \{0,1\}\}$, two
vertices $(x_1,x_2,...,x_n)$ and $(y_1,y_2,...,y_n)$ are adjacent if
and only if $x_i = y_i$ for all but one $i$. The folded hypercube
$FQ_n$ of dimension $n$, proposed first in $[1]$, is a graph
obtained from the hypercube $Q_n$ by adding an edge, called a
complementary edge, between any two vertices $ x=
(x_1,x_2,...,x_n)$, $y= (\bar{x_1},\bar{x_2},...,\bar{x_n})$, where
$\bar{1}=0$ and $\bar{0}=1$. The graphs shown in Fig. 1,  are the
folded hypercubes $ FQ_3$ and $FQ_4$.  The graphs $\Gamma_1 =
(V_1,E_1)$ and $\Gamma_2 = (V_2,E_2)$ are called isomorphic, if
there is a bijection $\alpha : V_1 \longrightarrow V_2 $   such
that, $\{a,b\} \in E_1$ if and only if $\{\alpha(a),\alpha(b)\} \in
E_2$ for all $a,b \in V_1$. In such a case the bijection $\alpha$ is
called an isomorphism. An automorphism of a graph $\Gamma $ is an
isomorphism of $\Gamma $ with itself. The set of automorphisms of
$\Gamma$, with the operation of composition of functions, is a
group, called the automorphism group of $\Gamma$ and denoted by $
Aut(\Gamma)$. A permutation of a set is a bijection of it with
itself.  The group of all permutations of a set $V$ is denoted by
$Sym(V)$, or just $Sym(n)$ when $\mid V \mid =n $. A permutation
group $G$ on $V$ is a subgroup of $Sym(V)$. In this case we say that
$G$ acts on $V$. If $\Gamma$ is a graph with vertex-set $V$, then we
can view each automorphism as a permutation of $V$, so $Aut(\Gamma)$
is a permutation group. Let $G$ acts on $V$, we say that $G$ is
transitive ( or $G$ acts transitively on $V$ ) if there is just one
orbit. This means that given any two elements $u$ and $v$ of $V$,
there is an element $ \beta $ of  $G$ such that $\beta (u)= v $.

The graph $\Gamma$ is called vertex transitive if  $Aut(\Gamma)$
acts transitively on $V(\Gamma)$. For $v\in V(\Gamma)$ and
$G=Aut(\Gamma)$, the stabilizer subgroup $G_v$ is the subgroup of
$G$ containing all automorphisms which fix $v$. In the vertex
transitive case all stabilizer subgroups $G_v $ are conjugate in
$G$, and consequently isomorphic, in this case, the index of $G_v$
in $G$ is given by the equation,  $ | G : G_v | =\frac{| G |}{| G_v
|} =| V(\Gamma)| $. If each stabilizer $ G_v $ is the identity
group, then every element of $G$, except the identity, does not fix
any vertex, and we say that $G$ acts semiregularly on $V$. We say
that $G$ acts regularly on $V$ if and only if $G$ acts transitively
and semiregularly on $ V$
 and in this case we
have $\mid V \mid = \mid G \mid$. The action of $Aut(\Gamma)$ on
$V(\Gamma)$ induces an action on $E(\Gamma)$  by the rule
$\beta\{x,y\}=\{\beta(x)$, $\beta(y)\}, \beta\in Aut(\Gamma)$, and
$\Gamma$ is called edge transitive if this action is transitive.The
graph $\Gamma$ is called symmetric, if for all vertices $u, v, x,
y,$ of $\Gamma$ such that $u$ and $v$ are adjacent, and $x$ and $y$
are adjacent, there is an automorphism $\alpha$ such that
$\alpha(u)=x,   and, \alpha(v)=y$. It is clear that a symmetric
graph is vertex transitive and edge transitive.

Let $G$ be any abstract finite group with identity $1$, and suppose
that $\Omega$ is a set of generators of $G$, with the properties :

(i) $x\in \Omega \Longrightarrow x^{-1} \in \Omega; (ii) 1\notin
\Omega $ ;

The Cayley graph $\Gamma=\Gamma (G, \Omega )$ is the graph whose
vertex-set and edge-set defined as follows : $V(\Gamma) = G  ;
E(\Gamma)=\{\{g,h\}\mid g^{-1}h\in \Omega \}$. \

It can be shown that the hypercube $Q_n$ is the Cayley graph
$\Gamma(Z_2^n,B)$, where $B=\{e_1,e_2,...,e_n\}$, $e_i$ is the
element of $Z_2^n$ with $1$ in the $i-th$ position and $0$ in the
other positions for, $1 \leq i \leq n$. Also, the folded hypercube
$FQ_n$ is the Cayley graph $\Gamma(Z_2^n,S)$, where $S=B\cup
\{u=e_1+e_2+...+e_n\}$. Hence the hypercube $Q_n$ and the folded
hypercube $FQ_n$ are vertex transitive graphs. Since $Q_n$ is
 Hamiltonian $[6]$ and  a spanning subgraph of $FQ_n$, so $FQ_n$ is Hamiltonian. Some
properties of the folded hypercube $FQ_n$ are discussed in
$[6,7,8]$. \

The group $G$ is called a semidirect product of $ N $ by $Q$,
denoted by $ G = N \rtimes Q $,
 if $G$ contains subgroups $ N $ and $ Q $ such that, (i)$
N \unlhd G $ ($N$ is a normal subgroup of $G$ ); (ii) $ NQ = G $;
(iii) $N \cap Q =\{1\} $. \ \

\

\

\

\

\

 \begin{figure}[h]
\def\emline#1#2#3#4#5#6{%
\put(#1,#2){\special{em:moveto}}%
\put(#4,#5){\special{em:lineto}}}
\def\newpic#1{}
%
%
%
\unitlength 0.3mm
\special{em:linewidth 0.4pt}
\linethickness{0.4pt}
\begin{picture}(-40,-40)(130,40)

%
\put(-23,59){\circle*{2}} \put(-22,-79){\circle*{2}}
\put(247,60){\circle*{2}} \put(245,-79){\circle*{2}}
\put(73,43){\circle*{2}} \put(158,43){\circle*{2}}
\put(35,0){\circle*{2}} \put(36,-26){\circle*{2}}
\put(193,0){\circle*{2}} \put(192,-28){\circle*{2}}
\put(-2,31){\circle*{2}} \put(-3,-55){\circle*{2}}
\put(224,30){\circle*{2}} \put(225,-49){\circle*{2}}
\put(17,127){\circle*{2}} \put(60,73){\circle*{2}}
\put(168,73){\circle*{2}} \put(211,129){\circle*{2}}
\put(72,93){\circle*{2}} \put(159,92){\circle*{2}}
\put(40,111){\circle*{2}} \put(188,111){\circle*{2}}
\put(73,-64){\circle*{2}} \put(157,-63){\circle*{2}}
\emline{-23}{59}{1}{-22}{-79}{2}
\emline{-23}{59}{1}{247}{60}{2}
\emline{-23}{59}{1}{73}{43}{2}
\emline{-23}{59}{1}{192}{-28}{2}
\emline{-23}{59}{1}{-2}{31}{2}
\emline{-22}{-79}{1}{245}{-79}{2}
\emline{-22}{-79}{1}{193}{0}{2}
\emline{-22}{-79}{1}{-3}{-55}{2}
\emline{-22}{-79}{1}{73}{-64}{2}
\emline{247}{60}{1}{245}{-79}{2}
\emline{247}{60}{1}{158}{43}{2}
\emline{247}{60}{1}{36}{-26}{2}
\emline{247}{60}{1}{224}{30}{2}
\emline{245}{-79}{1}{35}{0}{2}
\emline{245}{-79}{1}{225}{-49}{2}
\emline{245}{-79}{1}{157}{-63}{2}
\emline{73}{43}{1}{158}{43}{2}
\emline{73}{43}{1}{35}{0}{2}
\emline{73}{43}{1}{225}{-49}{2}
\emline{73}{43}{1}{73}{-64}{2}
\emline{158}{43}{1}{193}{0}{2}
\emline{158}{43}{1}{-3}{-55}{2}
\emline{158}{43}{1}{157}{-63}{2}
\emline{35}{0}{1}{36}{-26}{2}
\emline{35}{0}{1}{193}{0}{2}
\emline{35}{0}{1}{-2}{31}{2}
\emline{36}{-26}{1}{192}{-28}{2}
\emline{36}{-26}{1}{-3}{-55}{2}
\emline{36}{-26}{1}{73}{-64}{2}
\emline{193}{0}{1}{192}{-28}{2}
\emline{193}{0}{1}{224}{30}{2}
\emline{192}{-28}{1}{225}{-49}{2}
\emline{192}{-28}{1}{157}{-63}{2}
\emline{-2}{31}{1}{-3}{-55}{2}
\emline{-2}{31}{1}{224}{30}{2}
\emline{-2}{31}{1}{157}{-63}{2}
\emline{-3}{-55}{1}{225}{-49}{2}
\emline{224}{30}{1}{225}{-49}{2}
\emline{224}{30}{1}{73}{-64}{2}
\emline{17}{127}{1}{60}{73}{2}
\emline{17}{127}{1}{211}{129}{2}
\emline{17}{127}{1}{159}{92}{2}
\emline{17}{127}{1}{40}{111}{2}
\emline{60}{73}{1}{168}{73}{2}
\emline{60}{73}{1}{72}{93}{2}
\emline{60}{73}{1}{188}{111}{2}
\emline{168}{73}{1}{211}{129}{2}
\emline{168}{73}{1}{159}{92}{2}
\emline{168}{73}{1}{40}{111}{2}
\emline{211}{129}{1}{72}{93}{2}
\emline{211}{129}{1}{188}{111}{2}
\emline{72}{93}{1}{159}{92}{2}
\emline{72}{93}{1}{40}{111}{2}
\emline{159}{92}{1}{188}{111}{2}
\emline{40}{111}{1}{188}{111}{2}
\emline{73}{-64}{1}{157}{-63}{2}
\put(192,136){\makebox(0, 0)[cc]{001}} \put(165,114){\makebox(0,
0)[cc]{110}} \put(38,104){\makebox(0, 0)[cc]{111}}
\put(222,68){\makebox(0, 0)[cc]{1101}} \put(-19,65){\makebox(0,
0)[cc]{0010}} \put(152,52){\makebox(0, 0)[cc]{1100}}
\put(63,54){\makebox(0, 0)[cc]{0011}} \put(248,-67){\makebox(0,
0)[cc]{1001}} \put(0,40){\makebox(0, 0)[cc]{0000}}
\put(102,121){\makebox(0, 0)[cc]{FQ$_ 3$}} \put(1,126){\makebox(0,
0)[cc]{000}} \put(156,-69){\makebox(0, 0)[cc]{1000}}
\put(164,93){\makebox(0, 0)[cc]{100}} \put(70,88){\makebox(0,
0)[cc]{101}} \put(176,74){\makebox(0, 0)[cc]{011}}
\put(41,74){\makebox(0, 0)[cc]{010}} \put(65,-69){\makebox(0,
0)[cc]{0111}} \put(197,-22){\makebox(0, 0)[cc]{1010}}
\put(199,2){\makebox(0, 0)[cc]{1110}} \put(228,-42){\makebox(0,
0)[cc]{1011}} \put(226,23){\makebox(0, 0)[cc]{1111}}
\put(107,-9){\makebox(0, 0)[cc]{FQ$_ 4$}} \put(16,-23){\makebox(0,
0)[cc]{0101}} \put(-10,-68){\makebox(0, 0)[cc]{0110}}
\put(-21,-48){\makebox(0, 0)[cc]{0100}} \put(12,0){\makebox(0,
0)[cc]{0001}}

\end{picture}

\end{figure} \

\

 \

 \

\

\

 \ \ \ \ \ \ \ \ \ \ \ \ \ \ \ \ \ \ \ \  Fig. 1. The folded hypercubes $FQ_3$ and
 $FQ_4$.

\section{\bf Main results}
\begin{lem}\label{1} If $n \neq 3 $, then every 2-path in $FQ_n$  is contained in  a unique
4-cycle.

\end{lem}

\begin{proof}\ If  $n=2$, then it is trivial that the assertion of the Lemma
is true, so let $n>3$.  Let $P: uvw$ be a 2-path in $FQ_n$. If  $
u=(x_1,...,\bar{x_i},...,x_n)$, $v=(x_1,...,x_i,...,x_n)$,
$w=(x_1,...,\bar{x_j},...,x_n)$, then only vertex
$x=(x_1,...,x_{i-1}$, $\bar{x_i},...,x_{j-1},\bar{x_j},...,x_n)$ and
$v$ are adjacent to both vertices $u$ and $w$. Hence the 4-cycle $C:
uvwx $ is the unique 4-cycle that contains the 2-path $P$. If
$u=(\bar{x_1},\bar{x_2},...,\bar{x_n})$, $v=(x_1,...,x_i, ...,x_n)$,
$w=(x_1,...,x_{j-1},\bar{x_j},x_{j+1},...,x_n)$, then only vertices
$x=(\bar{x_1},...,\bar{x_{j-1}},x_j,\bar{x_{j+1}},...,\bar{x_n})$
and $v$ are adjacent to both vertices $u$ and $w$.

\end{proof}

   In the folded hypercube $FQ_3$ any 2-path is contained in 3
4-cycles, hence the assertion of Lemma 2.1 is not true in $FQ_3$.\

 \begin{rem}\ For a graph $\Gamma$ and $v\in V(\Gamma)$, let $N(v)$ be
the set of vertices $w$ of $\Gamma$ such that $w$ is adjacent to
$v$. Let $ G= Aut( \Gamma)$, then $G_v$ acts on $ N(v)$,  if we
restrict the domains of the permutations $g \in G_v $ to $ N(v)$.
Let $ L_v $ be the set of all elements $ g $ of $G_v$ such that $ g
$ fixes  each element of $ N(v)$. Let $ Y = N(v) $ and $ \Phi : G_v
\longrightarrow Sym(Y)$ be defined by the rule, $ \Phi(g) = g_{\mid
Y}$ for any element $g$ in $G_v$ , where $ g_{\mid Y}$ is the
restriction of $ g $ to $Y$. In fact $ \Phi $ is a  group
homomorphism and $ ker(\Phi) = L_v $,  thus $ G_v/L_v$ and the
 subgroup $ \phi(G_v)$ of $ Sym(Y) $ are isomorphic. If $\mid Y \mid = deg(v) = k$, then $
\mid G_v \mid / \mid L_v \mid \leq k! $. \

\end{rem}

\begin{lem}\label{1} If $n>3$ and $G=Aut(FQ_n)$, then
 $ | G | \leq (n+1)!2^n$

\end{lem}

\begin{proof}\ Let $ v \in V(FQ_n) $ and $ L_v $ be the subgroup which is defined in the
 above, we show that $ L_v = \{ 1\} $.  Let $g \in L_v $ and
$w$ be an arbitrary vertex of $FQ_n$. If the distance of $w$ from
$v$ is $1$, then $w$ is in $N(v)$, so $g(w)=w$. Let the distance of
$w$ from $v$ be $2$. Then there is a vertex $u$ such that $P : vuw $
is a 2-path, hence by Lemma 2.1.  there is a 4-cycle that contains
this 2-path, thus there is a vertex $t$ such that $C : tvuw $ is a
4-cycle. Since $ t \in L_v $, then $g(t)=t$, so $ g(C) : tvug(w)$ is
a 4-cycle. By Lemma 2.1 the 2-path $P_1 : tvu $ is contained in a
unique 4-cycle, thus $g(C) = C $, therefore $g(w) = w $. The set $S$
is a generating set for the Abelian group $Z_2^n$, so the Cayley
graph $FQ_n = \Gamma(Z_2^n, S) $ is a connected graph. Now, by
induction on the distance $w$ from $v$, it follows that $g(w)=w$, so
$g = 1 $ and $L_v=\{ 1\}$. Now,  by the Remark 2.2.  ,  $ | G_v |
\leq | L_v | (n+1)! \leq (n+1)!$. \

The folded hypercube $FQ_n$ is a vertex transitive graph, hence $ |
G |= | G_v | | V(FQ_n)| \leq (n+1)!2^n$.

\end{proof}

\begin{thm} \ If $n>3$ , then $Aut(FQ_n)$ is a semidirect
product of $N$ by $M$, where $N$ is isomorphic to the Abelian group
$Z_2^n$  and $M$ is isomorphic to the  group $Sym(n+1)$.

\end{thm}

\begin{proof}\ Let $Aut(FQ_n) =G$,  $v \in Z_2^n = V(FQ_n)$  and $\rho_v$ be
the mapping $ \rho_v : Z_2^n \longrightarrow Z_2^n $ defined by
$\rho_v(x)= v+x $. Since $FQ_n$ is the Cayley graph $\Gamma (Z_2^n,
S)$, then $\rho_v$ is an automorphism of $FQ_n$ and $N = \{\rho_v |
v \in Z_2^n \}$ is a subgroup of $G$ isomorphic to $Z_2^n$. Note
that the Abelian group $Z_2^n$ is also a vector space over the field
$ F=\{ 0,1 \}$ and $B=\{e_1,e_2,...,e_n\}$ is a basis of this vector
space and any subset of the set $ S= B \cup\{u=e_1+e_2+...+e_n\}$
with $n$ elements is linearly independent over $F$ and is a basis of
the vector space $Z_2^n$. Let $A$ be a subset of $S$ with $n$
elements and $ f : B \longrightarrow A $ be a one to one function.
We can extend $f$ over $Z_2^n$ linearly. Let $\phi$ be the linear
extension of $f$ over $Z_2^n$, thus $\phi$ is a linear mapping of the vector space
$Z_2^n$ into itself such that $\phi_{|B}= f$. Since $B$ and $A$ are bases of the
vector space $Z_2^n$, hence $\phi$ is a permutation of $Z_2^n$. In
fact $\phi$ is an automorphism of $FQ_n$. If $ A = B$, then $ \phi
(u)=\phi(e_1)+ \phi (e_2)+ ... +\phi (e_n)= e_1 + e_2 + ...+ e_n =
u$. If $A \neq B $, then $ u \in A $ and for some $ i,j \in
\{1,2,...,n \}$ we have $\phi (e_i) = u $ and $e_j \notin A$. Then
$\phi (u) = \phi(e_1)+ \phi(e_2)+...+\phi(e_n) = e_1 + e_2 + ...+
e_{j-1} +  e_{j+1} + ...+ e_n + u = u - e_j + u = e_j \in S$. Now,
it follows that $\phi$ maps $S$ into $S$. If $[v,w] \in E(FQ_n)$,
then $w= v+s $ for some $s \in S $, hence $ \phi(w)= \phi(v) +
\phi(s) $, now since $\phi(s) \in S$ we have $[\phi(v),\phi(w)] \in
E(FQ_n)$. For a fixed n-subset $A$ of $S$ there are $n!$ distinct
one to one functions such as $f$, thus there are $n!$ automorphisms
of the folded hypercube $FQ_n$ such as $\phi$. The set $S$ has
$n+1$ elements, so there are $n+1$ n-subset of $S$ such as $A$,
hence there are $(n+1)!$ one to one functions $ f: B\longrightarrow
S $. Let $M = \{\phi : Z_ 2^n \longrightarrow Z_ 2^n\  | \ \phi  $
is a linear extension of a one to one function $f:B \longrightarrow
S \}$.  Then $M$ has $(n+1)!$ elements and any element of $M$ is an
automorphism of $FQ_n$. If $ \alpha \in M$, then $ \alpha $ maps $S$
onto $S$, hence $ \alpha_{|S} $, the restriction of $ \alpha $ to
$S$, is a permutation of $S$. Now it is an easy task to show that
$M$ is isomorphic to the group $Sym(S)$. Every element of $M$
fixes the element $0$, thus $ N \cap M  = \{1\}$, hence $| MN | =
\frac{ | M | | N |}{| N \cap M |} = (2^n)(n+1)!$, therefore $ |
Aut(FQ_n) | \geq (2^n)(n+1)!$. Now, by the Lemma 2.3.  it follows
that $ | Aut(FQ_n)| = (n+1)!2^n$, therefore $Aut(FQ_n) = MN$. \

We show that the subgroup $N$ is a normal subgroup of $Aut(FQ_n)
=G=MN=NM$. It is enough to show that for any $f \in M$ and $g \in
N$,  we have $f^{-1}gf \in N$. There is an element $y \in Z_2^n$
such that $g= \rho_y$. Let $b$ be an arbitrary vertex of $FQ_n$,
then $f^{-1}gf(b)= f^{-1}\rho_yf(b)=f^{-1}(y+f(b))=f^{-1}(y)+ b=
\rho_{f^{-1}(y)}(b)$, hence $f^{-1}gf =\rho_{f^{-1}(y)} \in N $.

\end{proof}

   It is  an easy task to show that the folded hypercube $FQ_3$ is isomorphic to $
  K_{4,4}$, the complete bipartite graph of order 8, so $Aut(FQ_4)$
  is a group with $ 2(4!)^2 = 1152$ elements [2], therefore Theorem 2.3 is not true
  for $n=3$. \

  If $n>1$, then the assertion of Lemma 2.1 is also true for the
  hypercube $Q_n$ and by a similar method that has been seen in the
  proof of Theorem 2.4. we can show that $Aut(Q_n) \cong Z_2^n \rtimes
  Sym(n)$, the result which has been discussed in $[4]$ by a different
  method.

  \begin{thm} \ If $n \geq 2$, then the folded hypercube $ FQ_n $ is
  a symmetric graph.

  \end{thm}

  \begin{proof}\ The folded hypercube $FQ_2$ is isomorphic to $K_4$,
  the complete graph of order 4, and the folded hypercube $FQ_3$ is isomorphic to
  $K_{4,4}$, the complete bipartite graph of order 8, both of these
  are
  clearly symmetric. Let $ n \geq 4$. Since The folded hypercube
  $FQ_n$ is a Cayley graph, then it is vertex transitive, now it is
  sufficient to show that for a fixed vertex $v$ of $V(FQ_n)$, $G_v$
  acts transitively on $N(v)$, where $G = Aut(FQ_n) $. As we can
  see in the proof of Theorem 2.3, since each element of $M$ is a
  linear mapping of the vector space  ${Z_2}^n$ over $F = \{0,1 \}$, then for the vertex $ v = 0 $ the
  stabilizer group of $G_v$ is $M$. The restriction of each element
  of $M$ to $N(0)=S $
   is a permutation of $S$. If $f \in M $
  fixes each element of $S$, then $f$ is the identity mapping of
  the vector space ${Z_2}^n$. Since $ |S|= n+1, $ then $Sym(S)$ has
  $(n+1)!$ elements. On the other hand $\bar{M }= \{f_{|S}$ $ |  f \in M
  \}$ has $(n+1)!$ elements, hence $ \bar{M }= Sym(S) = G_0$. We
  know  that $Sym(X)$ acts transitively on $X$, where $X$ is a set,
  so $G_0$ acts transitively on $N(0)$.

   \end{proof}

   \begin{cor}The connectivity of the folded hypercube $FQ_n$ is
   maximum, say $n+1$.

\end{cor}

\begin{proof}\ Since the folded hypercube $FQ_n$ is a symmetric
graph,  then it is edge transitive, on the other hand this graph is
a regular graph of valency $n+1$. We know that the connectivity of a
connected edge transitive graph is equal to its minimum valency [3,
pp.  55].

\end{proof}

The above fact has been rephrased in $[1]$ and has been found in a
different manner. \

\

{\bf ACKNOWLEDGMENT}   \

 The author is grateful to professor Alireza Abdollahi and professor A. Mohammadi Hassanabadi for their  helpful comments and thanks the Center of Excellence for Mathematics, University of Isfahan.

\

\end{document}